\RequirePackage{fix-cm}
\documentclass[smallcondensed]{svjour3}     
\usepackage[text={14.6cm, 225truemm}, centering] {geometry}
\usepackage{amsmath}
\usepackage{ntheorem}
\usepackage{amsfonts}
\usepackage{amssymb}
\usepackage{cases}
\usepackage{enumerate, paralist}
\usepackage{xcolor}
\usepackage[all,poly,knot]{xy}
\usepackage[misc]{ifsym}   
\usepackage[numbers,sort&compress]{natbib}

\newcommand{\Z}{\mathbb{Z}}
\newcommand{\F}{\mathbb{F}}

\newtheorem{thm}{Theorem} 
\newtheorem{cor}{Corollary}
\newtheorem{lem}{Lemma}

\newtheorem{exam}{Example}
\newtheorem{rem}
{Remark}
\journalname{Designs, Codes and Cryptography}

\begin{document}

\title{Large classes of permutation polynomials over $\mathbb{F}_{q^2}$
\thanks{This work was supported in part by the National Natural Science Foundation of China
(Grant Nos. 11371106, 11271142, 61363069) and the Guangdong Provincial Natural
Science Foundation (Grant No. S2012010009942). }
}


\author{Yanbin Zheng \and Pingzhi Yuan \and Dingyi Pei}


\institute{Y. Zheng (\Letter)
              \at
              Guangxi Key Laboratory of Trusted Software, Guilin University of Electronic Technology,
              Guilin, China\\
              School of Computer Science and Engineering, Guilin University of Electronic Technology, Guilin, China\\
              \email{zhengyanbin16@126.com}           
           \and
           P. Yuan \at
              School of Mathematics, South China Normal University, Guangzhou, China\\
              \email{mcsypz@mail.sysu.edu.cn}
           \and
             D. Pei \and Y. Zheng\at
              School of Mathematics and Information Science, Guangzhou University, Guangzhou, China\\
              Key Laboratory of Mathematics and Interdisciplinary Sciences of Guangdong Higher
              Education Institutes, Guangzhou University, Guangzhou, China  \\
             \email{gztcdpei@scut.edu.cn}
}

\date{Received: date / Accepted: 27 December 2015}

\maketitle

\begin{abstract}
  Permutation polynomials (PPs) of the form $(x^{q} -x + c)^{\frac{q^2 -1}{3}+1} +x$ over $\mathbb{F}_{q^2}$  were presented by Li, Helleseth and Tang [Finite Fields Appl. 22 (2013) 16--23]. More recently, we have constructed PPs of the form $(x^{q} +bx + c)^{\frac{q^2 -1}{d}+1} -bx$ over $\mathbb{F}_{q^2}$, where $d=2, 3, 4, 6$ [Finite Fields Appl. 35 (2015) 215--230]. In this paper we concentrate our efforts on the PPs of more general form
  \[
   f(x)=(ax^{q} +bx +c)^r \phi((ax^{q} +bx +c)^{(q^2 -1)/d}) +ux^{q} +vx~~\text{over $\mathbb{F}_{q^2}$},
   \]
  where $a,b,c,u,v \in \mathbb{F}_{q^2}$, $r \in \mathbb{Z}^{+}$, $\phi(x)\in \mathbb{F}_{q^2}[x]$ and $d$ is an arbitrary positive divisor of $q^2-1$. The key step is the construction of a commutative diagram with specific properties, which is the basis of the Akbary--Ghioca--Wang (AGW) criterion. By employing the AGW criterion two times, we reduce the problem of determining whether $f(x)$ permutes $\mathbb{F}_{q^2}$ to that of verifying whether two more  polynomials permute two subsets of $\mathbb{F}_{q^2}$. As a consequence, we find a series of simple conditions for $f(x)$ to be a PP of $\mathbb{F}_{q^2}$. These results unify and generalize some known classes of PPs.
\end{abstract}

\section{Introduction}

For $q$ a prime power, let $\F_{q}$ be the finite field containing $q$ elements,
and $\F_{q}[x]$ the polynomial ring over $\F_{q}$.
A polynomial $f(x) \in \F_{q}[x]$ is called a permutation polynomial (PP) of $\F_{q}$ if
it permutes $\F_{q}$.
Many open problems about permutation polynomials (PPs) over finite fields can be found
in~\cite{LM88,LM93,open_Mullen}. One of these problems is how to construct new classes of PPs.
More recently, it has achieved significant progress; see for example,
\cite{AGW,CHZ14,Charpin,Dingpp,DingCode,hou,Hou12,Hou15,TZH,Wang-cyc,Wu_L-1,J.Yuan02,ZZH,ZH12,ZHC15,zyp-1,Zi09}.
A reason for such a rapid development is that PPs have some applications in
design theory~\cite{LidlFF,cmp}, coding theory~\cite{DingCode,code,SW10}, cryptography~\cite{crypto,crypto2,SSS11}, and other areas of science and engineering~\cite{LidlFF,HFF}.

Li, Helleseth and Tang presented a class of PPs of the form $(x^{q} -x + c)^{\frac{q^2 -1}{3}+1} +x$ over $\F_{q^2}$ by using the piecewise method~\cite{LHT13}. Later in~\cite{YZ15} we employed the Akbary--Ghioca--Wang (AGW) criterion to construct PPs of the form similar to
\[
P(x)=(x^{q} +bx + c)^{\frac{q^2 -1}{d}+1} -bx \text{~~over $\F_{q^2}$},
\]
where $d=2, 3, 4, 6$ and $d \mid q^2 -1$.
This paper is motivated by the question: when does $P(x)$ permute $\F_{q^2}$ for $d >6$?

In this paper we continue our investigations by using the AGW criterion, and
focus our attention on the PPs of the form
\begin{equation} \label{fx}
f(x)=(ax^{q} +bx +c)^r \phi((ax^{q} +bx +c)^{(q^2 -1)/d}) +ux^{q} +vx \text{~~over $\F_{q^2}$},
\end{equation}
where $a,b,c,u,v \in \F_{q^2}$, $r$ is a positive integer, $\phi(x)$ is an arbitrary polynomial over $\F_{q^2}$, and $d$ is an arbitrary positive divisor of $q^2-1$. Obviously, $f(x)$ is a generalization of $P(x)$.

\subsection{AGW criterion and the reduction of problem}

The following lemma was proved in~\cite{AGW} by Akbary, Ghioca and Wang,
which is referred to as the  AGW criterion in~\cite{HFF}.
For further results about the AGW criterion, we refer the reader to~\cite{YD-AGW,YD-AGW2,YZ15}.

\begin{lem}[AGW criterion]\label{thm_AGW}
Let $R$, $S$, $\bar{S}$ be finite sets with $\#S=\#\bar{S}$, and let $f: R\rightarrow R$,
$g: S\rightarrow\bar{S}$, $\theta: R\rightarrow S$ and $\bar{\theta}: R\rightarrow \bar{S}$
be mappings such that $\bar{\theta}\circ f= g\circ\theta$, i.e., the following diagram is commutative:
\[
\xymatrix{
  R\ar[d]_{\theta}\ar[r]^{f} & R\ar[d]^{\bar{\theta}}  \\
  S \ar[r]_{g}       & \bar{S} }
\]
If both $\theta$ and $\bar{\theta}$ are surjective, then the following statements are equivalent:
\begin{enumerate}[$(i)$]
  \item $f$ is bijective from $R$ to $R$ (permutes $R$).
  \item $g$ is bijective from $S$ to $\bar{S}$ and $f$ is injective on
        $\theta^{-1}(s)$ for each $s\in S$.
\end{enumerate}
\end{lem}

The AGW criterion gives us a recipe in which one can construct permutations of $R$ from bijections between two smaller sets $S$ and $\bar{S}$. But it requires two conditions $\bar{\theta}\circ f= g\circ\theta$ and $\#S=\#\bar{S}$. For $R=\F_{q^2}$ and the mapping induced by $P(x)$ with arbitrary $d$, it was a hard problem in the past for us to determine $\theta$, $\bar{\theta}$ and $g$ such that the previous two conditions are satisfied. This is the reason we only considered $P(x)$ with $d \le 6$ in~\cite{YZ15}.

The crucial breakthrough of this paper is that a commutative diagram with desired properties is constructed. More precisely, for $R=\F_{q^2}$ and the mapping $f$ induced by~\eqref{fx}, we find $\theta$, $\bar{\theta}$ and $g$ such that $\bar{\theta}\circ f= g\circ\theta$, $S=\bar{S}$ and $\#S=q$. Hence the upper half of Figure~\ref{bi-jht} is commutative. According to the AGW criterion,
the problem of determining whether $f(x)$ permutes $\F_{q^2}$ is reduced to another one: whether $g(x)$ permutes the smaller set $S$. The later problem is relatively simple since we need only calculate at most $q$ values to verifying it.

Based on the above result, we find $\lambda$ and $h$ such that $\lambda\circ g =h\circ\lambda$,
i.e., the lower half of Figure~\ref{bi-jht} is commutative.
By the AGW criterion, $g(x)$ permutes $S$ if and only if $h(x)$ permutes $U_{n}$,
where $n =d/\gcd(q+1,\,d)$ and $U_{n}$ is the set of $n$-th roots of unity in $\F_{q^2}$.
It is easy to check the permutation property of $h(x)$ on $U_{n}$ when $n$ is not very large.

\begin{figure}[ht!]
\[
\xymatrix{
  \F_{q^2} \ar[d]_{\theta} \ar[r]^{f} & \F_{q^2} \ar[d]^{\bar{\theta}} \\
  S \ar[d]_{\lambda} \ar[r]^{g} & \bar{S} \ar[d]^{\lambda} \\
  U_{n} \ar[r]^{h} & U_{n}   }
\]
\caption{Two commutative diagrams}\label{bi-jht}
\end{figure}
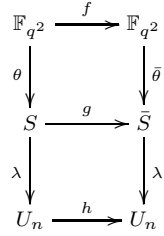
\subsection{Outline}
In Sections 2 and 3, for any $\phi(x) \in \F_{q^2}[x]$ we obtain three equivalent statements: $(i)$ $f(x)$ permutes $\F_{q^2}$; $(ii)$ $g(x)$ permutes  $S$; $(iii)$ $h(x)$ permutes~$U_{n}$.
From Section 4 to 7, we consider special forms of $\phi(x)$, and find a series of
simple conditions for $f(x)$ to be a PP of $\F_{q^2}$; for instance see Corollaries \ref{cor-4uxr}, \ref{cor-b=-1}, \ref{cor-d=6}, \ref{=1cor2} and~\ref{cor-fq2}. These results unify and generalize some PPs in~\cite{AGW,LHT13,YD-AGW,YZ15,ZH12}.

\section{The reduction from $\F_{q^2}$ to $S$}
For ease of notations, we 
let $A =bu -av$, $B =au^{q} -bv^{q}$ and $S =\{ax^q +bx +c \mid x \in \F_{q^2}\}$. 
We let $\Z^{+}$ denote the set of positive integers,
$\xi$ a primitive element of $\F_{q^2}$,
$U_{n}$ the set of $n$-th roots of unity in $\F_{q^2}$,
and $\F_{q}$ a subfield of $\F_{q^2}$.
In this section, we give the main result and its explanation.

\begin{thm}\label{thm_gene} 
Let $d \mid q^2 -1$, $r\in\Z^{+}$ and $\phi(x)\in\F_{q^2}[x]$. Let $a,b,c,u,v \in \F_{q^2}$
satisfy $ab \ne 0$, $a^{q+1} =b^{q+1}$ and $ac^{q} =b^{q}c$.
Then $f(x)$ in~\eqref{fx} is a PP of\, $\F_{q^2}$  if and only if\, $g(x)$ permutes $S$, where
\begin{equation}\label{eq-g}
g(x)=x^r\big[B\phi(x^{(q^2 -1)/d}) +A^{1-r}B^{qr}\phi(x^{(q^2 -1)/d})^{q}\big] +(u^{q+1}-v^{q+1})x.
\end{equation}
\end{thm}
\begin{proof}
To apply AGW criterion, we need to construct a commutative diagram.
Let 
\[
\theta(x) =ax^q +bx +c, \quad
\bar{\theta}(x)=Ax^q +Bx +C, \quad
\bar{S}=\{\bar{\theta}(x) \mid x \in \F_{q^2}\},
\]
where $C= (u^{q+1} -v^{q+1})c$. The proof is divided into four steps.

$(i)$ First, we prove that $\bar{\theta}\circ f= g\circ\theta$,
i.e., the following diagram is commutative:
 
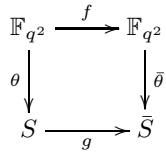
\[
  \xymatrix{
  \F_{q^2} \ar[d]_{\theta}\ar[r]^{f}  &  \F_{q^2}\ar[d]^{\bar{\theta}}  \\
  S        \ar[r]_{g}         &  \bar{S} }
 \]
For convenience, let $f_1(x)\equiv f_2(x)$ denote $f_1(x)\equiv f_2(x) \pmod{x^{q^2}-x}$. Since $a^{q+1} = b^{q+1}$ and $ac^{q} = b^{q}c$, we have
\begin{equation}\label{eq_image}
\begin{split}
a\theta(x)^q
&= a(ax^q +bx +c)^q
\equiv a(a^q x +b^q x^q +c^q) \\
&= ab^q x^q +a^{q+1} x +ac^q
= ab^q x^q +b^{q+1} x +b^qc
= b^q \theta(x).
\end{split}
\end{equation}
Therefore $a\theta(x)^{q} \equiv b^{q}\theta(x)$ and so $a^{q}\theta(x) \equiv b\theta(x)^{q}$. Furthermore,
\[\begin{split}
A\theta(x)^{q} &=(bu-av)\theta(x)^{q}=(b\theta(x)^{q})u -(a\theta(x)^{q})v \\
&\equiv(a^{q}\theta(x))u -(b^{q}\theta(x))v
=(a^{q}u -b^{q}v)\theta(x)= B^{q}\theta(x).
\end{split}\]
Hence
\begin{equation}\label{eq-rq}
A\theta(x)^{qr} =A^{1-r}(A\theta(x)^{q})^{r}
\equiv A^{1-r}(B^{q}\theta(x))^{r} =A^{1-r}B^{qr}\theta(x)^{r}.
\end{equation}
Let $\psi(x)= ux^{q} +vx$. Then
\begin{equation}\label{eq_bar_phi2}
\begin{split}
&\quad A\psi(x)^{q} +B\psi(x) +C\\
&= (bu -av)(ux^{q} +vx)^{q} +(au^{q} -bv^{q})(ux^{q} +vx) +(u^{q+1} -v^{q+1})c\\
&\equiv (bu -av)(u^{q}x +v^{q}x^{q}) +(au^{q} -bv^{q})(ux^{q} +vx) +(u^{q+1} -v^{q+1})c\\
&= (u^{q+1}-v^{q+1})ax^{q} +(u^{q+1}-v^{q+1})bx +(u^{q+1} -v^{q+1})c\\
&= (u^{q+1}-v^{q+1})(ax^{q} +bx +c)\\
&= (u^{q+1}-v^{q+1})\theta(x).
\end{split}
\end{equation}
Let $m=(q^2 -1)/d$. Then, by \eqref{fx}, $f(x) =\theta(x)^r\,\phi(\theta(x)^{m})+\psi(x)$ and,
by \eqref{eq-rq} and \eqref{eq_bar_phi2},
\[ \begin{split}
\bar{\theta}(f(x))
&= Af(x)^q +Bf(x) +C\\
&= A[\theta(x)^r\,\phi(\theta(x)^{m}) +\psi(x)]^{q} +B[\theta(x)^r\,\phi(\theta(x)^{m}) +\psi(x)] +C\\
&= A\theta(x)^{qr}\phi(\theta(x)^{m})^{q} +B\theta(x)^r\phi(\theta(x)^{m}) +A\psi(x)^{q} +B\psi(x) +C\\
&\equiv A^{1-r}B^{qr}\theta(x)^{r}\phi(\theta(x)^{m})^{q} +B\theta(x)^r\phi(\theta(x)^{m})
  +A\psi(x)^{q} +B\psi(x) +C\\
&= \theta(x)^r\big[ A^{1-r}B^{qr}\phi(\theta(x)^{m})^{q} +B\phi(\theta(x)^{m})\big] +(u^{q+1}-v^{q+1})\theta(x).
\end{split} \]
Since $$g(x)=x^r[A^{1-r}B^{qr}\phi(x^{m})^{q} +B\phi(x^{m})] +(u^{q+1}-v^{q+1})x,$$
we get $\bar{\theta}(f(x)) \equiv g(\theta(x))$, and so
$\bar{\theta}\circ f= g\circ\theta$.

$(ii)$ Now we show $S=\bar{S}$ if $A \ne 0$. It will be proved in Lemma~\ref{lem_Fq}
that $a^{q+1} = b^{q+1}$ if and only if $b = \xi^{(q-1)i}a^{q}$ for some $i\in\Z$.
According to \eqref{eq_image}, we obtain
\[ \theta(x)^{q} \equiv (b^q/a)\theta(x) =\xi^{-(q-1)i}\theta(x).\]
Hence $(\xi^{i}\theta(x))^{q} \equiv \xi^{i}\theta(x)$, and so
$\xi^{i}\theta(\alpha)\in \F_{q}$ for $\alpha \in \F_{q^2}$.
Thus there exists $e \in \F_{q}$ such that $\theta(\alpha) = \xi^{-i}e$,
i.e., $S \subseteq \{\xi^{-i}e \mid e \in \F_{q}\}$.
On the other hand, since $\theta(x)$ has degree~$q$, each image has at most $q$ preimages under $\theta$.
Hence $\#S\ge q^2 /\,q =q$, and so $S=\{ \xi^{-i}e \mid e\in\mathbb{F}_{q}\}$.
We next show that $A^{q+1}=B^{q+1}$ and $AC^{q}=B^{q}C$. Since
 \[\begin{split}
 A^{q+1}&=(bu-av)^{q}(bu-av)=b^{q+1}u^{q+1} -ab^qu^qv +(-1)^qa^qbuv^q +a^{q+1}v^{q+1},\\
 B^{q+1}&=(au^q-bv^q)^q(au^q-bv^q)=a^{q+1}u^{q+1} +(-1)^qab^qu^qv -a^qbuv^q +b^{q+1}v^{q+1},
 \end{split}\]
and $a^{q+1}=b^{q+1}$, we get $A^{q+1}=B^{q+1}$.
It follows from $ac^{q} =b^{q}c$ that $a^{q}c =bc^{q}$ and
\[\begin{split}
AC^{q}
&=(bu-av)(u^{q+1}-v^{q+1})^qc^q=(bc^qu-ac^qv)(u^{q+1}-v^{q+1})\\
&=(a^qcu-b^qcv)(u^{q+1}-v^{q+1})=(a^qu-b^qv)(u^{q+1}-v^{q+1})c=B^{q}C.
\end{split}\]
Similarly to \eqref{eq_image}, we obtain $A\theta(x)^q=B^q\theta(x)$. An argument similar to the one used above shows that $\bar{S}= \{\xi^{-i}e \mid e \in \F_{q}\}$ if $A \ne 0$. Therefore we conclude that
\begin{equation}\label{S}
S=\bar{S} =\{\xi^{-i}e \mid e \in \F_{q}\}.
\end{equation}

$(iii)$ Since $f(x) =\theta(x)^r \phi(\theta(x)^{m})+ ux^{q} +vx$, we have
\[
f(x)=\theta(x)^r \phi(\theta(x)^{m})+ (u/a)\theta(x) +(v -(bu/a))x -(cu/a).
\]
For any $x \in \theta^{-1}(s)$, $f(x)=(v -(bu/a)) x +t$, where $t=s^r \phi(s^{m})+ (u/a)s -(cu/a)$.
Thus $f(x)$ is injective on $\theta^{-1}(s)$ if and only if $v-(bu/a) \ne 0$, i.e., $A \ne 0$.

$(iv)$ If $A=0$ then $f(x)$ is not injective on $\theta^{-1}(s)$ and, by the AGW criterion, $f(x)$ is not a PP of $\F_{q^2}$. If $A=0$ then $B^{q+1} = A^{q+1}=0$ and $u^{q+1} = v^{q+1}$. Hence $g(x)=0$ and it does not permute~$S$.
Hence the theorem is true when $A=0$. If $A \ne 0$ then $f(x)$ is injective on $\theta^{-1}(s)$. By the AGW criterion and $S=\bar{S}$, $f(x)$ is a PP of $\F_{q^2}$ if and only if $g(x)$ permutes~$S$.
$\hfill{}\Box$ \end{proof}

In Theorem~\ref{thm_gene}, the problem of determining whether $f(x)$ permutes $\F_{q^2}$ is reduced to that of
establishing whether $g(x)$ permutes $S$. By~\eqref{S}, $S$ has $q$ elements. Thus we need only calculate at most $q$ values ($g(y)$ for $y\in S$) to verify whether $f(x)$ permutes $\F_{q^2}$. Hence Theorem~\ref{thm_gene} provides a relatively simple necessary and sufficient condition for $f(x)$ to be a PP of $\F_{q^2}$.

According to the commutative diagram in the proof of Theorem~\ref{thm_gene}, the polynomial $g(x)$ induces a mapping from $S$ to $S$. Hence $g(x)$ permutes $S$ if and only if $g(x)$ is injective on $S$,
which can be used to simplify the proof that $g(x)$ permutes $S$.
We next give further explanation for Theorem~\ref{thm_gene}.

\begin{lem}\label{lem_Fq}
Let $a, b, c \in \F_{q^2}$ with $ab \ne 0$. Then
\begin{enumerate}[$(i)$]
  \item $ax^{q} +bx +c$ is a PP of\, $\F_{q^2}$
        if and only if\, $a^{q+1} \neq b^{q+1}$.
  \item $a^{q+1} = b^{q+1}$ if and only if\, $b = \xi^{(q-1)i}a^{q}$ for some $i\in\Z$.
  \item If\, $b = \xi^{(q-1)i}a^{q}$, then $ac^{q} = b^{q}c$ if and only if\, $c =\xi^{-i}e$ for some $e \in \F_{q}$.
  \item The subfield\, $\F_{q} =\{0\}\cup\{\, \xi^{(q +1)j} \mid j= 1, 2,\dotsc, q-1 \}$.
\end{enumerate}
\end{lem}
\begin{proof}
$(i)$ $a x^{q} +b x$ is a PP of $\F_{q^2}$ if and only if
$\big|\begin{smallmatrix}
   \,a & b^q \\
   \,b  & a^q
\end{smallmatrix}\big| \ne 0$,
i.e., $a^{q+1} \ne b^{q+1}$.
$(ii)$ The norm function $\mathbf{N}_{\F_{q^2}/\F_{q}}(b/a^{q})=(b/a^{q})^{q+1} =1$
if and only if $b/a^q =e^{q-1}$ for some $e \in \F_{q^2}^{*}$.
Let $e=\xi^i$ for some $i\in\Z$. Then $b = \xi^{(q-1)i}a^{q}$.
$(iii)$ Since $a = \xi^{(q-1)i}b^{q}$, $ac^{q} = b^{q}c$ if and only if $(\xi^{i}c)^q = \xi^{i}c$,
i.e., $\xi^{i}c \in \F_{q}$. Hence $c =\xi^{-i}e$ for some $e \in \F_{q}$.
$(iv)$ Since $\xi$ is a primitive element of $\F_{q^{2}}$, $\xi^{(q +1)j}$ are all distinct for
$1 \le j \le q-1$. Also $(\xi^{(q +1)j})^q =\xi^{(q^2 +q)j} =\xi^{(q+1)j}$. Hence the result holds.
$\hfill{}\Box$ \end{proof}

According to Lemma~\ref{lem_Fq}~$(i)$, $S$ is a subset of $\F_{q^2}$,
i.e., $ax^{q} +bx +c$ is not a PP of $\F_{q^2}$, if and only if $a^{q+1} = b^{q+1}$.
From $(ii)$, if $b=a^q$ then $S=\F_q$ and $f(x)$ permutes $\F_{q^2}$ if and only if $g(x)$ permutes $\F_q$.
By $(ii)$ and $(iii)$, there are $(q^2-1)\cdot(q+1)\cdot q$ choices for $a,b,c \in \F_{q^2}$
such that $ab \ne 0$, $a^{q+1}  =b^{q+1}$  and $ac^{q} =b^{q}c$.
$(iv)$ describes the construction of the subfield $\F_{q}$ of $\F_{q^2}$,
which plays an important role in the proof of Theorem~\ref{thm_UN}.

\section{The reduction from $S$ to $U_n$}

Theorem \ref{thm_gene} requires that $b = \xi^{(q-1)i}a^{q}$ for some $i\in\Z$.
When $i=j\cdot\gcd(q+1,\,d)$, we obtain the following result by employing the AGW criterion again.

\begin{thm}\label{thm_UN}
Let $d \mid q^2-1$ and $d=n\cdot\gcd(q+1, d)$.
Let $a,b,c \in \F_{q^2}$, $j\in\Z$ satisfy $ab \ne 0$, $b =\xi^{(q-1)jd/n}a^{q}$ and $ac^{q} =b^{q}c$. Let $u$, $v \in \F_{q^2}$, $r\in\Z^{+}$ satisfy $u^{q+1}=v^{q+1}$ and $\gcd(r,(q^2-1)/d)=1$, or $u^{q+1} \ne v^{q+1}$ and $r=1$. Then $f(x)$ in~\eqref{fx} is a PP of $\F_{q^2}$  if and only if\, $h(x)$ permutes $U_{n}$, where
\begin{equation}\label{h}
   h(x)= x^r[B\phi(x) +A^{1-r}B^{qr}\phi(x)^{q} +u^{q+1} -v^{q+1}]^{(q^2-1)/d}.
\end{equation}
\end{thm}

\begin{proof}
Let $g(x)$ be defined by~\eqref{eq-g}, $\lambda(x)=x^{(q^2 -1)/d}$ and $T=\{\lambda(x) \mid x \in S\}$. Since all the conditions of Theorem~\ref{thm_gene} are satisfied, we need only prove that $g(x)$ permutes $S$ if and only if $h(x)$ permutes $T$. The proof is divided into four steps.

$(i)$ When $u^{q+1} =v^{q+1}$ or $u^{q+1}  \ne v^{q+1} $ and $r =1$, it is easy to show that
 $\lambda\circ g= h\circ\lambda$, i.e., the following diagram is commutative:
    \[
     \xymatrix{
      S\ar[d]_{\lambda}\ar[r]^{g}  & S\ar[d]^{\lambda}  \\
      T \ar[r]_{h}                      & T }
    \]

$(ii)$
We now prove $T= \{0\}\cup U_{n}$. By~\eqref{S}, if $b =\xi^{(q-1)jd/n}a^{q}$ then
$S = \{\,\xi^{-jd/n}e \mid e \in \F_{q}\}$.
Because $\F_{q} =\{0\}\cup\{\, \xi^{(q +1)i} \mid i= 1, 2,\dotsc, q-1 \}$, we have
\begin{equation}\label{S1}
S =\{0\} \cup \{\,\xi^{(q+1)i -jd/n} \mid i=1,2,\dotsc,q-1\}.
\end{equation}
Since $\gcd(q+1,\,d)=d/n$, there is $m \in \Z$ such that $q+1=md/n$ and $\gcd(m, n)=1$. Hence
\[
(q+1)i -jd/n=(mi-j)d/n
\]
and $\{mi-j \mid i= 1, 2,\dotsc, n\}$ is a complete residue system modulo $n$.
Because $d \mid q^2 -1$, namely $(n d/n) \mid (md/n)(q-1)$, we get $n \mid q-1$. Thus
\begin{equation}\label{S2}
S =\{0\}\cup D'_{0} \cup D'_{1}\cup \dotsm \cup D'_{n-1},
\end{equation}
where $D'_{i}$ is a set of $(q-1)/n$ elements from $D_{i}$, and
\[
D_{i} =\{ \xi^{kd +id/n}
\mid k=0,1,\dotsc, \tfrac{q^2 -1}{d}-1 \}, \quad 0 \le i \le n-1.
\]
For $x \in D_i$, we have $\lambda(x)=x^{(q^2-1)/d}=\omega^i$, where $\omega =\xi^{(q^2-1)/n}$. Hence
\begin{equation}\label{eq_Im_d/n}
\{\lambda(x) \mid x \in S\} =\{\,0,\omega^0, \omega^1, \omega^2, \dotsc, \omega^{n-1}\},
\quad \text{i.e.,} \quad T =\{0\} \cup U_{n}.
\end{equation}

$(iii)$ It follows from~\eqref{S2} and \eqref{eq_Im_d/n} that $\lambda^{-1}(\omega^i)=D'_{i}$.
We show next that $g(x)$ is injective on $D'_{i}$ when $\gcd(r,(q^2-1)/d)=1$.
For $x \in  D_{i}$, $\lambda(x)=\omega^i$ and so
\[ g(x)=\left\{
\begin{array}{ll}
    x^r[B\phi(\omega^i) +A^{1-r}B^{qr}\phi(\omega^i)^{q}] & \text{ for $u^{q+1} =v^{q+1}$,} \\
    x^r[B\phi(\omega^i) +B^{q}\phi(\omega^i)^{q} +u^{q+1}-v^{q+1}] & \text{ for \,$u^{q+1}  \ne v^{q+1}$, $r=1$}.
\end{array}\right.\]
It is easy to prove that the following statements are equivalent:
\begin{enumerate}
  \item $g(x)$ is injective on $D_{i}$.
  \item $x^r$ is injective on $D_i$.
  \item $\{rk \mid k=0,1,\dotsc, \tfrac{q^2 -1}{d}-1 \}$ is a complete residue system modulo $(q^2-1)/d$.
  \item $\gcd(r,(q^2-1)/d)=1$.
\end{enumerate}
Since $D'_i \subseteq D_i$, $g(x)$ is injective on $D'_{i}$ when $\gcd(r,(q^2-1)/d)=1$.

$(iv)$ By the AGW criterion, $g(x)$ permutes $S$ if and only if $h(x)$ permutes~$T$.
The later is equivalent to that $h(x)$ permutes $U_{n}$ because $T=\{0\} \cup U_{n}$ and $h(0)=0$.
$\hfill{}\Box$ \end{proof}

In the proof above, the problem of determining whether $g(x)$ permutes $S$ is reduced to that of verifying whether $h(x)$ permutes $U_{n}$. Hence Theorem~\ref{thm_UN} indicates that $f(x)$ in \eqref{fx} permutes $\F_{q^2}$  if and only if $h(x)$ permutes $U_{n}$. In this case we need only calculate at most $n$ values ($h(y)$ for $y\in U_{n}$) to verify whether $f(x)$ is a PP of $\F_{q^2}$. Therefore, roughly speaking, Theorem~\ref{thm_UN} is more efficient than Theorem~\ref{thm_gene} when $n<q$. In particular, if $n=1$ and $h(1)=1$ then $f(x)$ permutes~$\F_{q^2}$.

Let $d \mid q^2-1$. Then $d$ is an odd divisor of $q-1$ if and only if $\gcd(q+1,d)=1$.
From Lemma~\ref{lem_Fq}, $a^{q+1} = b^{q+1}$ if and only if $b=\xi^{(q-1)i}a^{q}$ for some $i\in\Z$.
Consequently we obtain the following result by applying Theorem~\ref{thm_UN} with $n=d$.

\begin{cor}\label{cor-odd}
Let $d \geq 3$ be an odd divisor of\, $q-1$. Let $a,b,c \in \F_{q^2}$ satisfy $ab \ne 0$, $a^{q+1} =b^{q+1}$ and $ac^{q} =b^{q}c$. Let $r,u,v$ be as in Theorem~\ref{thm_UN}. Then $f(x)$ in~\eqref{fx} is a PP of $\F_{q^2}$  if and only if\, $h(x)$ in~\eqref{h} permutes $U_{d}$.
\end{cor}

We next consider the case that $d$ is an even divisor of $q-1$.

\begin{cor}\label{thm-even}
Let $d \geq 2$ be an even divisor of $q-1$. Let $a,b,c \in \F_{q^2}$, $j\in\Z$ satisfy $ab\ne 0$, $b =\xi^{(q-1)j}a^{q}$ and $ac^{q} =b^{q}c$. Let $r,u,v$ be as in Theorem~\ref{thm_UN}. Then $f(x)$ in~\eqref{fx} is a PP of $\F_{q^2}$  if and only if
\begin{enumerate}[$(i)$]
  \item $h(x)$ in~\eqref{h} permutes $U_{d/2}$ when $j$ is even, or
  \item $h(x)$ in~\eqref{h} permutes $U_{d}\setminus U_{d/2}$ when $j$ is odd.
\end{enumerate}
\end{cor}
\begin{proof}
Since $d \mid q-1$ and $d$ is even, $q+1 \equiv 2 \pmod{d}$ and $\gcd(q+1,\,d)=\gcd(2,d)=2$.
For even $j$, the conditions in Theorem~\ref{thm_UN} are satisfied.
Hence $n=d/2$ and $U_{n} =U_{d/2}$, and so $(i)$ is true.
For odd $j$, an argument similar to that leading to~\eqref{eq_Im_d/n} shows
\[
\{x^{(q^2-1)/d} \mid x \in S\}
=\{\,0, \omega, \omega^3, \omega^5,\dotsc, \omega^{d-1}\}=\{0\}\cup U_{d}\setminus U_{d/2},
\]
where $\omega =\xi^{(q^2-1)/d}$. The remaining proof is analogous to that in Theorem~\ref{thm_UN}.
$\hfill{}\Box$ \end{proof}

When $d \mid q-1$, Corollaries~\ref{cor-odd} and \ref{thm-even} give necessary and sufficient conditions for $f(x)$ to be a PP of $\F_{q^2}$. In Corollary~\ref{thm-even}, let $d=6$, $q=13$ and $a=1$. If $j=q+1=14$ then $b=1$. If $j=(q+1)/2=7$ then $b=-1$. Thus we obtain the following example.
\begin{exam}
Let $\phi(x) \in\F_{169}[x]$, $r \ge 1$ and $\gcd(r,28)=1$. Define
\[
f_1(x)=(x^{13}+x)^r\phi((x^{13}+x)^{28}) +x^{13}-x,
\]
\[
f_2(x)=(x^{13}-x)^r\phi((x^{13}-x)^{28}) +x^{13}+x,
\]
\[
h(x)=2x^r(\phi(x)+\phi(x)^{13})^{28}.
\]
Then $f_1(x)$ is a PP of $\F_{169}$ if and only if $h(x)$ permutes $U_{3}=\{1,3,9\}$,
and $f_2(x)$ is a PP of $\F_{169}$ if and only if $h(x)$ permutes $U_{6}\setminus U_{3}=\{4,12,10\}$.
\end{exam}

We conclude this section with a remark on the degree of $\phi(x)$. Because
\[
x^r(x^{(q^2-1)/d})^d \equiv x^r \pmod{x^{q^2}-x},
\]
we need only consider $\deg(\phi(x))<d$ and $1 \le r \le (q^2-1)/d$ in Theorem~\ref{thm_gene}.
By~\eqref{eq_Im_d/n},
\begin{equation}\label{abc-un}
\{(ax^{q} +bx +c)^{(q^2-1)/d} \mid x \in \F_{q^2}\}=\{0\}\cup U_n.
\end{equation}
Hence we need only consider $\deg(\phi(x))<n$ in Theorem~\ref{thm_UN} and Corollary~\ref{cor-odd}.

\section{$\phi(x)$ is a constant}

In Sections 2 and 3, the polynomial $\phi(x)$ is arbitrary. In order to find simple conditions for $f(x)$ to be a PP, special forms of $\phi(x)$ are considered from Section~4 to~7.
We first investigate the case that $\phi(x)$ is a constant.
\subsection{The case $\phi(x)=1$}\label{sub-h=1}
When $\phi(x)=1$, we deduce the next result from Theorem~\ref{thm_gene}.
\begin{thm}\label{thm-r}
Let $a,b,c,u,v\in \F_{q^2}$  satisfy $ab \ne 0$, $a^{q+1}  =b^{q+1}$  and $ac^{q}  =b^{q}c$. Define
\begin{equation*}\label{eq-r}
\begin{split}
  f(x) & =(ax^{q} +bx +c)^r + ux^{q} +vx, \\
  g(x) & =(B +A^{1-r}B^{qr})x^r +(u^{q+1}-v^{q+1})x,
\end{split}
\end{equation*}
where $r\in\Z^{+}$. Then $f(x)$ is a PP of $\F_{q^2}$ if and only if $g(x)$ permutes $S$.
In particular, if $B +A^{1-r}B^{qr}=0$ then $f(x)$ permutes $\F_{q^2}$ if and only if $u^{q+1} \ne v^{q+1}$.
If $u^{q+1}  = v^{q+1}$ then $f(x)$ permutes $\F_{q^2}$ if and only if $B +A^{1-r}B^{qr} \ne 0$ and $\gcd(r,q-1)=1$.
\end{thm}
\begin{proof}
The first part is a direct consequence of Theorem~\ref{thm_gene}. Note that $g(x)$ permutes $S$ if and only if $g(x)$ is injective on $S$. If $B +A^{1-r}B^{qr}=0$ then $g(x)=(u^{q+1}  -v^{q+1}) x$.
Thus $g(x)$ is injective on $S$ if and only if $u^{q+1} \ne v^{q+1}$.
If $u^{q+1}  = v^{q+1}$ then $g(x)=(B +A^{1-r}B^{qr}) x^r$. Note that
\[
S=\{\xi^{-i}e \mid e \in \F_{q}\} \text{\quad and \quad}
\F_{q}=\{0\}\cup\{ \xi^{(q +1)j} \mid j= 1, 2,\dotsc, q-1 \}.
\]
Hence $g(x)$ is injective on $S$ if and only if $B +A^{1-r}B^{qr} \ne 0$ and $\{rj \mid j= 1,2,\dotsc,q-1\}$
is a complete residue system modulo $q-1$, i.e., $\gcd(r, q-1) =1$.
$\hfill{}\Box$ \end{proof}

Theorem~\ref{thm-r} unifies and generalizes several recently discovered PPs. Applying Theorem~\ref{thm-r} to $a=b=1$, we arrive at the following examples.

\begin{cor}\label{cor-4uxr}
Let $q$ be an odd prime power, $r\in\Z^{+}$ and $c, u \in \F_{q}$. Then
\[ f(x)=(x^{q} +x +c)^{r} +ux^{q} -ux \]
is a PP of $\F_{q^2}$  if and only if\, $u \ne 0$ and $\gcd(r, q-1) =1$.
\end{cor}

\begin{cor}\label{cor-a=b=1}
Let $r\in\Z^{+}$, $c, u \in \F_{q}$ and $e \in \F_{q^2}$ with $e+e^q=0$. Then
\begin{equation*}\label{eq-r1}
  f(x)=(x^{q} +x +c)^r + ux^{q} +(u-e)x
\end{equation*}
is a PP of $\F_{q^2}$  if and only if\, $e \ne 0$.
\end{cor}

Corollary \ref{cor-a=b=1} generalizes \cite[Theorem 5]{ZH12}, which only gives the sufficient part.

\begin{cor}
Let $m,r\in\Z^{+}$, $q=2^m$ and $c, u, e \in \F_{q}$ with $e \ne 0$ or $1$. Let
\[
  f(x)=(x^{q} +x +c)^r + ux^{q} +(u+e)x.
\]
Then both $f(x)$ and $f(x)+x$ are PPs of $\F_{q^2}$.
\end{cor}

A polynomial $f(x)$ is called a complete permutation polynomial (CPP) if both $f(x)$ and $f(x)+x$ are PPs. CPPs are useful in the construction of orthogonal latin squares~\cite{cmp}, check digit systems~\cite{SW10} and bent-negabent functions~\cite{SGCGS12}. If we change $v$ to $v+1$ in some results of this paper, we can obtain more CPPs.

Applying Theorem~\ref{thm-r} to $a=1$ and $b=-1$, we obtain the next corollaries.
\begin{cor}\label{cor-b=-1}
Let $r\in\Z^{+}$, $c, u, v \in \F_{q^2}$ satisfy $c+c^q=0$ and $(u+v)^q =(-1)^{r}(u+v)$. Then
\begin{equation*}\label{eq-b=-1}
  f(x)=(x^{q} -x +c)^r + ux^{q} +vx
\end{equation*}
is a PP of $\F_{q^2}$ if and only if\, $u^{q+1}  \ne v^{q+1}$.
\end{cor}

Corollary~\ref{cor-b=-1} generalizes \cite[Theorem 4]{ZH12}, which considers special $u$ and $v$ and only gives the sufficient part.

\begin{cor}\label{cor2-b=-1}
Let $r\in\Z^{+}$, $c, u, v \in \F_{q^2}$ satisfy $c+c^q=0$ and $u^{q+1} = v^{q+1}$. Then
\begin{equation*}\label{eq2-b=-1}
  f(x)=(x^{q} -x +c)^r + ux^{q} +vx
\end{equation*}
is a PP of $\F_{q^2}$ if and only if\, $\gcd(r,q-1)=1$ and $(u+v)^q =(-1)^{r}(u+v)$.
\end{cor}

Corollaries~\ref{cor-b=-1} and \ref{cor2-b=-1} unify and generalize \cite[Theorem 5.12]{AGW} and \cite[Theorem~6.4]{YD-AGW}, which requires $c=0$ and $u=v$ or $v^q$. Moreover, the first statement of Theorem~8.1.71 in the Handbook of Finite Fields \cite{HFF} is incorrect since the condition is sufficient but not necessary. Using the fact that $v-u^q \in \F_{q}$ and $u^{q+1}-v^{q+1}=(u+v)(u^q-v)$, we can show that the sufficient part is covered by Corollary~\ref{cor-b=-1}. For $v\in \F_4 \setminus \F_2$, the PP $(x^2+x)+x^2+vx$ of $\F_4$ is a counterexample of the necessary situation.

\subsection{The case $d \mid q+1$}

When $d \mid q+1$, for any $\phi(x)\in \F_{q^2}[x]$ we infer the next result which is similar to Theorem~\ref{thm-r}.
\begin{thm}\label{thm-d-q+1}
Under the hypotheses of Theorem~\ref{thm_gene} with $d \mid q^2-1$ replaced by $d \mid q+1$,
the polynomial $f(x)$ in~\eqref{fx} satisfies the congruence
\[
f(x) \equiv \phi(\delta)(ax^{q} +bx +c)^r +ux^{q} +vx \pmod{x^{q^2} -x},
\]
where $\delta=(b^q/a)^{(q+1)/d}$. If $\phi(\delta)\ne 0$ then $f(x)$ is a PP of $\F_{q^2}$ if and only if
\[
g(x)=(\bar{B} +\bar{A}^{1-r}\bar{B}^{qr})x^r +(\bar{u}^{q+1}-\bar{v}^{q+1})x
\]
permutes $S$, where $\bar{A} =b\bar{u} -a\bar{v}$,
$\bar{B} =a\bar{u}^{q} -b\bar{v}^{q}$, $\bar{u} =u/\phi(\delta)$ and $\bar{v} =v/\phi(\delta)$.
\end{thm}

\begin{proof}
Let $\theta(x)= ax^{q} +bx +c$. Then $f(x)= \theta(x)^r \phi(\theta(x)^{(q^2-1)/d}) +ux^{q} +vx$.
By~\eqref{eq_image},
\[ \theta(x)^{q} \equiv (b^q/a)\theta(x) \pmod{x^{q^2} -x}.\]
For any $e \in \F_{q^2}$, if $\theta(e) \ne 0$ then $\theta(e)^{q-1} = b^q/a$. Hence
\[
 \theta(e)^r \phi(\theta(e)^{(q^2-1)/d})
=\theta(e)^r \phi(\theta(e)^{(q-1)(q+1)/d})
=\theta(e)^r \phi((b^q/a)^{(q+1)/d}),\]
\[
\theta(x)^r \phi(\theta(x)^{(q^2-1)/d})
\equiv \theta(x)^r \phi((b^q/a)^{(q+1)/d}) \pmod{x^{q^2} -x}.
\]
Denote $\delta =(b^q/a)^{(q+1)/d}$. Then $f(x) \equiv \theta(x)^r\phi(\delta) +ux^{q} +vx \pmod{x^{q^2} -x}$.

The second part of the theorem is a direct consequence of Theorem~\ref{thm-r}.
$\hfill{}\Box$ \end{proof}

Applying Theorem~\ref{thm-d-q+1} to $r=1$, we obtain the following result.
\begin{cor}\label{d-q+1}
Under the hypotheses of Theorem~\ref{thm_gene} with $d \mid q^2-1$ and $r\in\Z^{+}$ replaced by $d \mid q+1$ and $r=1$, the polynomial $f(x)$ in~\eqref{fx} satisfies the congruence
\[
f(x) \equiv \alpha x^{q} +\beta x +\gamma \pmod{x^{q^2} -x},
\]
where $\alpha =\phi(\delta)a +u$, $\beta =\phi(\delta)b +v$, $\gamma =\phi(\delta)c$ and $\delta=(b^q/a)^{(q+1)/d}$.
Moreover, $f(x)$ is a PP of $\F_{q^2}$ if and only if\, $\alpha^{q+1} \ne \beta^{q+1}$.
\end{cor}

Corollary~\ref{d-q+1} reduces $f(x)$ in~\eqref{fx} to a simple linearized polynomial. Theorem~8 in~\cite{LHT13} and some theorems in~\cite{YZ15} are special cases of Corollary~\ref{d-q+1}.

\begin{exam}\label{Li8}
Let $q$ be an odd prime power with $3 \mid q+1$, and let $c\in \F_{q^2}$ with $c+c^q=0$.
Theorem 8 in~\cite{LHT13} pointed out that
\[ f(x)=(x^q -x +c)^{\frac{q^2-1}{3}+1} +x\] permutes $\F_{q^2}$.
In fact, by Corollary~$\ref{d-q+1}$, we have
\[ f(x) \equiv \alpha x^q +(1-\alpha)x +\alpha c \pmod{x^{q^2} -x}, \]
where $\alpha=(-1)^{(q+1)/3}$. Since $\alpha^{q+1} \ne (1-\alpha)^{q+1}$, it follows that $f(x)$ is a PP of $\F_{q^2}$.
\end{exam}

\section{$\phi(x)$ has degree less than three}\label{sec_n=2}

According to Theorem \ref{thm_UN}, the permutation behavior of $f(x)$ in \eqref{fx} over $\F_{q^2}$ is easy to check when $n$ is small. For $n=2$ or $3$, it suffices to consider $\deg(\phi(x))<2$ or $3$. In this case, we derive  explicit classes of PPs of $\F_{q^2}$. These results unify and generalize several classes of PPs in~\cite{LHT13,YZ15}.

\subsection{The case $n=2$}

\begin{lem}\label{lem-n=2}
Let $d \ge 4$ be even. Then the following two conditions are equivalent:
\begin{enumerate}[$(i)$]
  \item $d \mid q^2-1$, $\gcd(q+1, \,d)=d/2$;
  \item $d \equiv 0 \pmod{4}$, $q + 1 \equiv d/2 \pmod{d}$.
\end{enumerate}
\end{lem}
\begin{proof}
Clearly, $\gcd(q+1, \,d)=d/2$ if and only if $q+1 \equiv d/2 \pmod{d}$.
Hence \[ q-1 \equiv (d-4)/2 \pmod{d}, \quad\quad q^2-1 \equiv d^2/4 \pmod{d}, \]
and so $d \mid q^2-1$ if and only if $d \mid (d^2/4)$, i.e., $d =4s$ for some $s\in\Z$.
$\hfill{}\Box$ \end{proof}

Combining Theorem~\ref{thm_UN} and Lemma~\ref{lem-n=2}, we deduce the next result.

\begin{thm}\label{thm_d/2}
Let $d \equiv 0 \pmod{4}$ and $q + 1 \equiv d/2 \pmod{d}$. Let $a,b,c\in \F_{q^2}$, $j\in\Z$ satisfy $ab\ne 0$, $b =\xi^{(q-1)j d/2}a^{q}$ and $ac^{q} =b^{q}c$.
Let $r,u,v$ be as in Theorem~\ref{thm_UN} and
\begin{equation*}\label{eq-d/2}
f(x)=(a x^{q} +b x +c)^r\big(e_{0} +e_{1}(a x^{q} +b x +c)^{(q^2-1)/d}\big) + u x^{q} +v x,
\end{equation*}
where $e_{0},e_{1} \in \F_{q^2}$. Then $f(x)$ is a PP of $\F_{q^2}$  if and only if
\[
\big[ (Be_{0} +A^{1-r}B^{qr}e_{0}^q +u^{q+1} -v^{q+1})^2 -(Be_{1} +A^{1-r}B^{qr}e_{1}^q)^2 \big]^{(q^2-1)/d} =(-1)^{r+1}.
\]
\end{thm}
\begin{proof}
Clearly $\phi(x)=e_{0} +e_{1}x$. By Theorem~\ref{thm_UN}, $f(x)$ is a PP of $\F_{q^2}$  if and only if
\[
h(x) = x^r\big[ B(e_{0} +e_{1}x) +A^{1-r}B^{qr}(e_{0} +e_{1}x)^{q} +u^{q+1} -v^{q+1}\big]^{(q^2-1)/d}
\]
permutes $U_2=\{-1,1\}$. According to the commutative diagram in the proof of Theorem~\ref{thm_UN}, $h(x)$ induces a mapping from $\{0,-1,1\}$ to itself. Hence $h(x)$ permutes $U_2$ if and only if $h(-1)h(1) =-1$.
A simple calculation can complete the proof.
$\hfill{}\Box$ \end{proof}

Taking $d=4$ and $r=1$ in Theorem~\ref{thm_d/2}, we obtain the following example.
\begin{exam}
Let $d =4$ and $q+1\equiv 2  \pmod{4}$. Let $b, c\in \F_{q^2}$, $j\in\Z$ satisfy
$b=\xi^{(q-1)2j}$ and $c^q =b^q c$. Let $e, e' \in \F_{q}$ and
\[f(x) = e(x^{q} +b x +c)^{\frac{q^2-1}{4} +1} +e'(x^{q} +b x +c)^{\frac{3(q^2-1)}{4} +1} -bx.\]
It follows from~\eqref{abc-un} that
$\{(x^{q} +bx +c)^{(q^2-1)/4} \mid x \in \F_{q^2}\}=\{0,-1,1\}$. Hence
\[ (x^{q} +b x +c)^{\frac{3(q^2-1)}{4}+1} \equiv (x^{q} +b x +c)^{\frac{q^2-1}{4}+1} \pmod{x^{q^2} -x}, \]
\[f(x) \equiv (e+e')(x^{q} +b x +c)^{\frac{q^2-1}{4} +1} -bx \pmod{x^{q^2} -x}. \]
Clearly $b^{q+1}=1$ and $B=(-b)^{q+1}=1$. 
From Theorem~\ref{thm_d/2}, $f(x)$ is a PP of $\F_{q^2}$ if and only if $[1 -4(e+e')^2]^{(q^2-1)/4} =1$.
The result is Theorem $5.3$ in~\cite{YZ15}.
\end{exam}

\subsection{The case $n=3$}
The next result is a particular case for $n=3$ and $\phi(x)=x$ or $x^2$.

\begin{thm}\label{thm-d/3}
Let $3 \mid d$, $d \mid q^2-1$ and $\gcd(q+1,\,d) =d/3$. Let $b,c \in \F_{q^2}$, $j\in\Z$ satisfy $b=\xi^{(q-1)jd/3}$ and $c^{q} =b^{q}c$. Define
\[
f(x) = (x^{q} +b x +c)^{\frac{k(q^2-1)}{d}\,+1} -bx, \text{~where $k=1, 2$.}
\]
Let $\omega=\xi^{(q^2-1)/3}$. Then $f(x)$ is a PP of\, $\F_{q^2}$  if and only if
\[ \big( (\omega^{kq}+\omega^{k} -1)^{\frac{q^2-1}{d}},\;\;
         (\omega^{2kq}+\omega^{2k} -1)^{\frac{q^2-1}{d}} \big)= (1,1) \text{~or~} (\omega, \,\omega^{2}).
\]
\end{thm}
\begin{proof}
Clearly $B=(-b)^{q+1}=1$ and $\phi(x)=x^k$. According to Theorem~\ref{thm_UN}, $f(x)$ is a PP of $\F_{q^2}$
if and only if $h(x)=x(x^{kq}+x^{k}-1)^{\frac{q^2-1}{d}}$ permutes $U_3=\{1,\omega, \omega^2\}$. Since $h(1)=1$,
$h(x)$ permutes $U_3$ if and only if $(h(\omega),~ h(\omega^2))= (\omega,\omega^2)$ or $(\omega^2,\omega)$.
$\hfill{}\Box$ \end{proof}

\begin{rem}
Let $3 \mid d$ and $\gcd(q+1, \,d)=d/3$. $(1)$ If\, $q+1 \equiv d/3 \pmod{d}$,
then $d \mid q^2-1$ if and only if\, $d \equiv 6\pmod{9}$.
$(2)$ If\, $q+1 \equiv 2d/3 \pmod{d}$, then $d \mid q^2-1$ if and only if\, $d \equiv 3 \pmod{9}$.
\end{rem}

Theorem~\ref{thm-d/3} unifies and generalizes some results in~\cite{YZ15}.
In fact, if $d=3$ and $q+1 \equiv 2 \pmod{3}$, then $3 \mid q-1$ and $\omega^{q}=\omega$.
Therefore we have the next corollaries.

\begin{cor}\label{cor-d=3}
Let $3 \mid q-1$, and let $b$, $c \in \F_{q^2}$ satisfy $b^{q+1}=1$ and $c^{q} =b^{q}c$. Define
\[
f(x) = (x^{q} +b x +c)^{\frac{k(q^2-1)}{3}\,+1} -bx, \text{~where $k=1, 2$.}
\]
Let $\omega=\xi^{(q^2-1)/3}$. Then $f(x)$ is a PP of\, $\F_{q^2}$  if and only if
\[
\big( (2\omega^{k}-1)^{\frac{q^2-1}{3}},\;\;
         (2\omega^{2k} -1)^{\frac{q^2-1}{3}} \big)= (1,1)\text{~or~}(\omega, \,\omega^{2}).
\]
\end{cor}

Taking $k=1$, Corollary~\ref{cor-d=3} reduces to Theorem $4.5$ in~\cite{YZ15}.
Applying Theorem~\ref{thm-d/3} to $d=6$ and $q+1 \equiv 2 \pmod{6}$, we have the following result.

\begin{cor}\label{cor-d=6}
Let $q+1 \equiv 2 \pmod{6}$, and let $b,c \in \F_{q^2}$, $j\in\Z$ satisfy $b=\xi^{(q-1)2j}$ and $c^{q} =b^{q}c$. Define
\[
f(x) = (x^{q} +b x +c)^{\frac{k(q^2-1)}{6}\,+1} -bx, \text{~where $k=1, 2$.}
\]
Let $\omega=\xi^{(q^2-1)/3}$. Then $f(x)$ is a PP of\, $\F_{q^2}$  if and only if
\[ \big( (2\omega^{k}-1)^{\frac{q^2-1}{6}},\;\;
         (2\omega^{2k} -1)^{\frac{q^2-1}{6}} \big)= (1,1) \text{~or~} (\omega, \,\omega^{2}).
\]
\end{cor}

Corollary~\ref{cor-d=6} unifies Theorems $6.3$ and $6.4$ in~\cite{YZ15}.

\section{$\phi(x)= 1+ x +x^2 +\dotsm +x^{d-1}$}\label{sec_h=1}

In this section we consider the case $\phi(x)=\sum_{k =0}^{d-1}x^{k}$ and $r=1$.
If $d=2$ and $d \mid q^2 -1$ then $d \mid q+1$ and, by Corollary~\ref{d-q+1}, $f(x)$ in~\eqref{fx}
is equivalent to a simple linearized polynomial.
Hence we assume $d \geq 3$ in this section. Let $\omega =\xi^{(q^2-1)/d}$. Then
\begin{equation}\label{eq_h=1}
 \phi(\omega^i)=\left\{
\begin{array}{ll}
    0 & \text{ for \,$i =1,2,\dotsc,d-1$}, \\
    d & \text{ for \,$i = d$}.
\end{array}\right.
\end{equation}
The following theorem is based on this observation.

\begin{thm}\label{thm=1} 
Let $d \geq 3$, $d \mid q^2 -1$, and $a,b,c,u,v\in \F_{q^2}$ satisfy
$ab \ne 0$, $a^{q+1}= b^{q+1}$, $ac^{q} =b^{q}c$ and $u^{q+1}\ne v^{q+1}$. Define
\begin{equation}\label{=1}
f(x) =\sum_{k =0}^{d-1}(a x^{q} + b x + c)^{\frac{k\,(q^2-1)}{d}\, +1}+ u x^{q} + v x.
\end{equation}
Then the following statements hold:
\begin{enumerate}[$(i)$]
  \item The polynomial $f(x)$ is a PP of $\F_{q^2}$ if
       \begin{equation}\label{iff=1}
      \bigg(1+ \dfrac{d(B +B^{q})}{u^{q+1} -\,v^{q+1}} \bigg)^{(q^2 -1)/d} =1.
    \end{equation}
  \item For $d$ an odd divisor of $q-1$, $f(x)$ is a PP of $\F_{q^2}$ if and only if \eqref{iff=1} holds.
\end{enumerate}
\end{thm}
\begin{proof}
It follows from Theorem~\ref{thm_gene} that $f(x)$ is a PP of $\F_{q^2}$  if and only if
\[
g(x)=x [B\phi(x^{(q^2 -1)/d}) +B^{q}\phi(x^{(q^2 -1)/d})^{q} +u^{q+1}-v^{q+1}]
\]
permutes $S$, where $\phi(x)=\sum_{k =0}^{d-1}x^k$. Let
\[
D_1=\{x\in S \mid x\ne 0, x^{(q^2 -1)/d}= 1\}, \quad D_2=\{x\in S \mid x\ne 0, x^{(q^2 -1)/d} \ne 1\}.
\]
Then $S=\{0\} \cup D_1\cup D_2$. It follows from \eqref{eq_h=1} that
\[
\phi(x^{(q^2-1)/d})=\left\{\begin{array}{ll}
    d  & \text{ for $x \in D_1$,} \\
    0  & \text{ for $x \in D_2$,}
\end{array}\right. \text{\quad }
\begin{gathered}
g(x)=\left\{\begin{array}{ll}
    \alpha x  & \text{ for $x \in D_1$,} \\
    \beta x   & \text{ for $x \in D_2$,}
\end{array}\right.
\end{gathered} \]
where $ \alpha =d(B +B^{q}) +u^{q+1} -v^{q+1}$ and $\beta= u^{q+1} -v^{q+1}\ne 0$.

$(i)$ If \eqref{iff=1} holds, i.e., $(\alpha/\beta)^{(q^2 -1)/d} =1$, then
$\alpha^{(q^2 -1)/d} =\beta^{(q^2 -1)/d}$. Therefore
\[
\{\alpha x \mid x \in D_1\} \cap \{\,\beta x \mid x \in D_2\} =\emptyset,
\]
and so $g(x)$ is injective on $S$. Because $g(x)$ induces a mapping from $S$ to $S$, we deduce that $g(x)$ permutes $S$. Furthermore, $f(x)$ is a PP of $\F_{q^2}$.

$(ii)$ If $d$ is an odd divisor of $q-1$, by Corollary~\ref{cor-odd}, $f(x)$ is a PP of $\F_{q^2}$  if and only if
\[
h(x) = x[ B\phi(x) +B^{q}\phi(x)^{q} +u^{q+1} -v^{q+1}]^{(q^2-1)/d}
\]
permutes $U_d=\{1, \omega, \omega^{2},\dotsm, \omega^{d-1}\}$, where $\omega=\xi^{(q^2-1)/d}$.
The formula~\eqref{eq_h=1} leads to
\[
h(x)=\left\{\begin{array}{ll}
    \alpha      & \text{ for $x =1$, }\\
    \beta x   & \text{ for $x =\omega, \omega^{2},\dotsm, \omega^{d-1}$,}
\end{array}\right.
\]
where $\alpha =( d(B +B^{q}) +u^{q+1} -v^{q+1})^{(q^2 -1)/d}$
and $\beta =(u^{q+1} -v^{q+1})^{(q^2 -1)/d}\ne 0$.
As $\beta \in U_d$, we have $\{\alpha, \beta\omega, \dotsm, \beta\omega^{d-1} \} = U_d$
if and only if $\{\alpha/\beta, \omega, \dotsm, \omega^{d-1} \} = U_d$.
Hence $h(x)$ permutes $U_d$ if and only if $\alpha/\beta =1$.
$\hfill{}\Box$ \end{proof}

Recall that $B =au^{q} -bv^{q}$, if $u=a$ and $v=0$, then $B =B^q=a^{q+1}$. If $u=0$ and $v=-b$, then $B =B^q=b^{q+1}$. Thus we obtain the next result.

\begin{cor}\label{u=a}
Under the hypotheses of Theorem~\ref{thm=1}, take $u=a$ and $v=0$.
Then $f(x)$ in \eqref{=1} is a PP of $\F_{q^2}$ if $(1+2d)^{(q^2 -1)/d}=1$.
Take $u=0$ and $v=-b$. Then $f(x)$ in \eqref{=1} is a PP of $\F_{q^2}$ if $(1-2d)^{\frac{q^2 -1}{d}}=1$.
\end{cor}

We can also derive the following classes of PPs from Theorem~\ref{thm=1}.

\begin{cor}\label{=1cor2}
For $q$ an odd prime power, let $d \ge 3$ be an odd divisor of\, $q-1$, let $b$, $c\in \F_{q^2}$
satisfy $b^{q+1} =1$ and $c^{q} =b^{q}c$, and let $e \in \F_{q}$ with $e \ne 0$. Then
\[
f(x) =\sum_{k =0}^{d-1}e(x^{q} +bx +c)^{\frac{k(q^2-1)}{d}+1} +(1-e)x^{q} - b(1+e) x
\]
is a PP of $\F_{q^2}$ if and only if\, $(1-d)^{(q^2 -1)/d} =1$.
\end{cor}
\begin{proof}
Clearly  $a=1$, $u =(1-e)/e$ and $v = -b(1+e)/e$. Then $B= 2/e$, $u^{q+1}-v^{q+1}=-4/e$  and
$1+d(B^{q} +B)/(u^{q+1}\, -v^{q+1}) =1-d$.
$\hfill{}\Box$ \end{proof}

Corollary \ref{=1cor2} generalizes Theorem $4.9$ in~\cite{YZ15}, where $d=3$ is considered.

\begin{cor}\label{=1cor3}
Let $d\ge 3$ be an odd divisor of\, $q-1$, and $b,c\in \F_{q^2}$
satisfy $b^{q+1} =1$ and $c^{q} =b^{q}c$. Let $e \in \F_{q}$ with $e(1+2e)\ne 0$. Then
\[
f(x) =\sum_{k =0}^{d-1}e(x^{q} +bx +c)^{\frac{k(q^2-1)}{d}+1} -ex^{q} -b(1+e)x
\]
is a PP of $\F_{q^2}$ if and only if\, $\Big(1-\dfrac{2de}{1+2e}\Big)^{(q^2 -1)/d} =1$.
\end{cor}

Corollary \ref{=1cor3} generalizes Theorem $4.10$ in~\cite{YZ15}, where $d=3$ is considered.

\section{$\phi(x) \in \F_{q}[x]$}\label{sec_d_q-1}

The polynomial $\phi(x)$ over $\F_{q}$ is considered in this section. The fact that $\phi(x)^q =\phi(x^q)$ is employed to reduce the polynomials $g(x)$ in \eqref{eq-g} and $h(x)$ in \eqref{h}.

\subsection{The reduced form of $g(x)$}

Let $\phi(x) \in \F_{q}[x]$. Then $\phi(x)^q =\phi(x^q)$.  If $d \mid q-1$, we have
\[
(x^{(q^2-1)/d})^{q} \equiv  x^{(q^2-1)/d} \pmod{x^{q^2}-x}.
\]
Also note that
\[
g(x)= x^{r}\big[ B\phi(x^{(q^2-1)/d}) +A^{1-r}B^{qr}\phi(x^{(q^2-1)/d})^{q}\big] +(u^{q+1} -v^{q+1})x.
 \]
Therefore
\[
g(x) \equiv (B +A^{1-r}B^{qr})x^{r}\phi(x^{(q^2-1)/d}) +(u^{q+1} -v^{q+1})x \pmod{x^{q^2}-x}.
\]
If $B +A^{1-r}B^{qr}=0$ then
\[
g(x) \equiv (u^{q+1} -v^{q+1})x \pmod{x^{q^2}-x}.
\]
Thus $g(x)$ permutes $S$ if and only if $u^{q+1} \ne v^{q+1}$. By Theorem~\ref{thm_gene}, we get the next result.

\begin{thm}\label{thm1-hq=h}
Let $\phi(x) \in \F_{q}[x]$, $d \mid q-1$ and $r\in\Z^{+}$.
Let $a,b,c,u,v \in \F_{q^2}$ satisfy $ab \ne 0$, $a^{q+1} =b^{q+1}$,
$ac^{q} =b^{q}c$ and $B +A^{1-r}B^{qr}=0$. Then $f(x)$ in~\eqref{fx}
is a PP of $\F_{q^2}$ if and only if $u^{q+1} \ne v^{q+1}$.
\end{thm}

Theorem~\ref{thm1-hq=h} gives explicit conditions in which $f(x)$ is a PP of $\F_{q^2}$.
We derive the following classes of PPs from Theorem~\ref{thm1-hq=h}.
\begin{cor}\label{cor-fq}
Let $\phi(x) \in \F_{q}[x]$, $d \mid q-1$ and $r\in\Z^{+}$.
Let $c, u \in \F_{q}$ and $e \in \F_{q^2}$ with $e+e^q=0$. Then
\[
  f(x)=(x^{q} +x +c)^r\phi((x^{q} +x +c)^{(q^2 -1)/d})  + ux^{q} +(u-e)x
\]
is a PP of $\F_{q^2}$  if and only if\, $e \ne 0$.
\end{cor}

Corollary \ref{cor-fq} generalizes Corollary \ref{cor-a=b=1}, which only considers the case $\phi(x)=1$.

\begin{cor}\label{cor-fq2}
Let $\phi(x) \in \F_{q}[x]$, $d \mid q-1$ and $r$ be even.
Let $u, e \in \F_{q}$ and $c \in \F_{q^2}$ with $c+c^q=0$.  Then
\[
  f(x)=(x^{q} -x +c)^r \phi((x^{q} -x +c)^{(q^2 -1)/d}) + ux^{q} +(u+e)x
\]
is a PP of $\F_{q^2}$ if and only if\, $e(2u+e) \ne 0$.
\end{cor}

Corollary \ref{cor-fq2} generalizes \cite[Theorem 4]{ZH12},
which considers the case $\phi(x)=1$ and only gives the sufficient part.

\subsection{The reduced form of $h(x)$}

Let $u^{q+1} =v^{q+1}$. Then, by $\phi(x)^q =\phi(x^q)$, the polynomial $h(x)$ in \eqref{h} reduces to
\[
 h(x) =x^r[B\phi(x) +A^{1-r}B^{qr} \phi(x^q)]^{(q^2-1)/d}.
\]
If $n \mid q-1$ then $x^q \equiv x$ for $x \in U_n$, and so we may substitute $x$ for $x^q$ in $h(x)$. Hence
\[
 h(x) =(B +A^{1-r}B^{qr})^{(q^2-1)/d} x^r\phi(x)^{(q^2-1)/d}.
\]
Since $d \mid q^2-1$ and $d=n\cdot\gcd(q+1, d)$ imply $n \mid q-1$, we deduce the next result.

\begin{thm}\label{fqx}
Under the hypotheses of Theorem~\ref{thm_UN}, take $\phi(x) \in \F_{q}[x]$ and $u^{q+1} =v^{q+1}$.
Then $f(x)$ in~\eqref{fx} is a PP of $\F_{q^2}$ if and only if $B +A^{1-r}B^{qr} \ne 0$ and $x^r\phi(x)^{(q^2-1)/d}$ permutes $U_{n}$.
\end{thm}

In Theorem~\ref{fqx}, let $q=11$, $d=15$, $n=5$, $j=2$, $a=-b=u=v=1$ and $r=3$. Then $-A=B=2$ and $B +A^{1-r}B^{qr} =4\ne 0$. Hence we obtain the following example.
\begin{exam}
Let $\phi(x) \in \F_{11}[x]$ and $c \in \F_{121}$ with $c+c^{q}=0$. Then
\[
(x^{11}-x+c)^3 \phi((x^{11}-x+c)^8) +x^{11} +x
\]
is a PP of $\F_{121}$ if and only if $x^3\phi(x)^8$ permutes $U_{5}=\{1,4,5,9,3\}$.
\end{exam}

\section{Conclusions}

By constructing two commutative diagrams, we present large classes of permutation polynomials over $\F_{q^2}$. These results unify and generalize some families of permutation polynomials. The idea and technology of this paper is expected to be extended to other finite fields; for example, $\F_{q^3}$ or $\F_{q^n}$. These permutation polynomials may have many practical applications, including constructions of cryptographic functions, cyclic codes, difference sets, and sequences with good randomness.

\begin{acknowledgements}
We are grateful to the two anonymous referees for useful comments and suggestions.
\end{acknowledgements}



\end{document}